\documentclass[11pt]{article}
\PassOptionsToPackage{obeyspaces}{url}
\usepackage{amsfonts, amsmath, amssymb}
\usepackage[colorlinks=true,citecolor=blue,urlcolor=blue,linkcolor=blue,bookmarksopen=true,hidelinks]{hyperref}
\usepackage{amsthm}

\usepackage{breakurl}
\usepackage{tikz}
\usetikzlibrary{decorations.pathmorphing,calc}

\usepackage{etoolbox}
\patchcmd{\thebibliography}{\leftmargin\labelwidth}{\leftmargin\labelwidth\addtolength\itemsep{-0.1\baselineskip}}{}{}

\usepackage{mathtools}

\usepackage{enumitem}

\newcommand*\samethanks[1][\value{footnote}]{\footnotemark[#1]}

\author{Boris Bukh\thanks{Department of Mathematical Sciences, Carnegie Mellon University, Pittsburgh, PA 15213, USA\@. Supported in part by U.S.\ taxpayers through NSF CAREER grant DMS-1555149. Email: {\tt bbukh@math.cmu.edu}, {\tt tchao2@andrew.cmu.edu}}
\and
Ting-Wei Chao\samethanks
\and
Ron Holzman\thanks{Department of Mathematics, Technion-Israel Institute of Technology, 3200003 Haifa, Israel. Work done during a visit at the Department of Mathematics, Princeton University, supported by the H2020-MSCA-RISE project CoSP--GA No.\ 823748. Email: {\tt holzman@technion.ac.il}}}

\title{On convex holes in $d$-dimensional point sets}
\date{}

\usepackage[nameinlink]{cleveref}

\newtheorem{theorem}{Theorem}
\newtheorem{lemma}[theorem]{Lemma}

\newtheorem{proposition}[theorem]{Proposition}
\theoremstyle{definition}
\newtheorem{definition}[theorem]{Definition}

\newcommand*{\eqdef}{\stackrel{\mbox{\normalfont\tiny def}}{=}}   
\newcommand*{\veps}{\varepsilon}                                  
\DeclarePairedDelimiter\abs{\lvert}{\rvert}                       
\newcommand*{\R}{\mathbb{R}}                                      
\newcommand*{\N}{\mathbb{N}}                                      
\newcommand*{\prf}{\preceq}                                       
\newcommand*{\convo}{\conv^o}                                     
\newcommand*{\ba}{\mathbf{a}}
\newcommand*{\bb}{\mathbf{b}}                                     
\newcommand*{\bx}{\mathbf{x}}
\newcommand*{\by}{\mathbf{y}}
\newcommand*{\bz}{\mathbf{z}}
\newcommand*{\da}{\hat{a}}                                        
\DeclareMathOperator{\conv}{conv}                                 
\DeclareMathOperator{\len}{len}                                   
\DeclareMathOperator{\dist}{dist}                                 

\newcommand*{\doi}[1]{\href{http://dx.doi.org/#1}{doi:#1}}

\begin{document}

\maketitle

\begin{abstract} Given a finite set $A \subseteq \mathbb{R}^d$, points $a_1,a_2,\dotsc,a_{\ell} \in A$ form an $\ell$-hole in $A$ if they are the vertices of a convex polytope which contains no points of $A$ in its interior. We construct arbitrarily large point sets in general position in $\mathbb{R}^d$ having no holes of size $O(4^dd\log d)$ or more. This improves the previously known upper bound of order $d^{d+o(d)}$ due to Valtr. The basic version of our construction uses a certain type of equidistributed point sets, originating from numerical analysis, known as $(t,m,s)$-nets or $(t,s)$-sequences, yielding a bound of $2^{7d}$. The better bound is obtained using a variant of $(t,m,s)$-nets, obeying a relaxed equidistribution condition.
\end{abstract}

\section{Introduction}
A finite set $A \subseteq \mathbb{R}^d$ is in \emph{general position} if any $k$-dimensional affine subspace of $\mathbb{R}^d$, with $k < d$, contains at most $k+1$ points of $A$. Points $a_1,a_2,\dotsc,a_{\ell} \in A$ are in \emph{convex position} if they are the vertices of a convex polytope. If that polytope is empty, i.e., contains no points of $A$ in its interior, the points $a_1,a_2,\dotsc,a_{\ell}$ are said to form an $\ell$-\emph{hole} in $A$.

A classic result of Erd\H{o}s and Szekeres~\cite{ES} asserts that for any positive integer $\ell$, every sufficiently large finite set $A$ in general position in $\mathbb{R}^2$ contains $\ell$ points in convex position. Erd\H{o}s~\cite{E} went on to ask if one can also guarantee an $\ell$-hole in a large enough $A \subseteq \mathbb{R}^2$ in general position. Harborth~\cite{Ha} proved that one can always find a $5$-hole, while Horton~\cite{Ho} constructed arbitrarily large sets without any $7$-hole. The remaining case $\ell=6$ turned out to be more challenging, but was settled in the affirmative by Nicol\'as~\cite{N} and, independently, Gerken~\cite{G}.

Another question studied is the asymptotic behavior, as $n \to \infty$, of the number of $\ell$-holes guaranteed to exist in a set $A$ of $n$ points in general position in $\mathbb{R}^2$. For $\ell=3,4$ this number was shown to be $\Theta(n^2)$ by Katchalski and Meir~\cite{KM} and B\'ar\'any and F\"uredi~\cite{BF}. The order of magnitude for $\ell=5,6$ is not known, but very recently Aichholzer, Balko, Hackl, Kyn\v{c}l, Parada, Scheucher, Valtr and Vogtenhuber~\cite{Aet} proved it is superlinear for $\ell=5$.

Turning to higher dimensions, much less is known. Valtr~\cite{V} gave a simple projection argument to extend the Erd\H{o}s--Szekeres result to any dimension $d \ge 2$: for every $\ell$, any sufficiently large finite set $A$ in general position in $\mathbb{R}^d$ contains $\ell$ points in convex position. Regarding holes, he defined:\[h(d) \eqdef \max\{\ell:\textrm{any large enough }A \subseteq \mathbb{R}^d\textrm{ in general position contains an }\ell\textrm{-hole}\}.\] Using this notation, the $2$-dimensional results recalled above say that $h(2)=6$. Valtr proved the following bounds for $d \ge 3$:\[2d+1 \le h(d) \le 2^{d-1}(P(d-1)+1),\] where $P(d-1)$ is the product of the smallest $d-1$ prime numbers (and thus is asymptotically $d^{d+o(d)}$). For $d=3$ he gave the better upper bound $h(3) \le 22$. These remained the best known bounds on $h(d)$ for almost 30 years. In this note, we improve the upper bound to become exponential in $d$.

\begin{theorem} \label{thm:upper-bound}
For all $d \ge 3$ we have $h(d) < 2^{7d}$.
\end{theorem}

In fact, the upper bound that we get, as explained below, is slightly better, with the exponent reduced from $7d$ to less than $7d - 8\sqrt{\frac{d-2}{3}}$. For low values of $d$, we get even better bounds, e.g.,\[h(3) \le 32,\,h(4) \le 240,\,h(5) \le 988,\,h(6) \le 8000,\] which (except for $d=3$) improve upon those of Valtr.

In order to explain the source of our improvement, we recall that, generalizing Horton's original $2$\nobreakdash-dimensional construction, Valtr constructed for any $d \ge 2$ arbitrarily large point sets in $\mathbb{R}^d$, which he called $d$\nobreakdash-Horton sets, containing no hole of size greater than $2^{d-1}(P(d-1)+1)$. The key property that he used, and the one responsible for the superexponential term $P(d-1)$ in the bound, is the following: For two relatively prime moduli $q_1$ and $q_2$ and any two residue classes $r_1(\operatorname{mod} q_1)$ and $r_2(\operatorname{mod} q_2)$, their intersection is equidistributed in the sense that it contains one of any $q_1q_2$ consecutive integers (by the Chinese remainder theorem). We generalize Horton's construction in a different way, using another kind of equidistribution which is ``cheaper'' to achieve. Instead of recruiting larger and larger prime factors as the dimension grows, we use the fixed prime $2$. The relevant notion of equidistribution is captured by the following definition, due to Sobol'~\cite{So}.

\begin{definition}
Let $t \le m$ be nonnegative integers, and let $s$ be a positive integer.

A subset $X \subseteq [0,1)^s$ is a \emph{$(t,m,s)$-net} in base $2$ if $\abs{X} = 2^m$ and every dyadic sub-box $B$ of $[0,1)^s$ of the form
\[B = \prod_{i=1}^s \Bigl[\frac{b_i}{2^{k_i}}, \frac{b_i + 1}{2^{k_i}}\Bigr),\] where $b_i, k_i$ are nonnegative integers, $b_i < 2^{k_i}$, and $\sum_{i=1}^s k_i = m-t$, contains exactly $2^t$ points of $X$.

For a real number $y \in [0,1]$ let $y = \sum_{j=1}^{\infty} \frac{y_j}{2^j}$ with $y_j \in \{0,1\}$ be a binary expansion of $y$, and $[y]_m = \sum_{j=1}^m \frac{y_j}{2^j}$ its length $m$ truncation (which may depend on the choice of expansion). For $x \in [0,1]^s$ we write $[x]_m$ for the point in $[0,1]^s$ obtained by applying this truncation coordinatewise.

An infinite sequence $x_0,x_1,\ldots$ of points in $[0,1)^s$ with prescribed binary expansions of their coordinates is a \emph{$(t,s)$-sequence} in base $2$ if for every nonnegative integer $a$ and every integer $m > t$, the set $X_{a,m} \subseteq [0,1]^s$ given by \[X_{a,m} = \{[x_n]_m:\,a2^m \le n < (a+1)2^m\}\] is a $(t,m,s)$-net in base $2$.
\end{definition}

These notions (and their analogs in bases other than $2$) have been studied intensively in discrepancy theory, with applications to numerical analysis. The goal is, for a given dimension $s$, to construct $(t,s)$\nobreakdash-sequences and hence $(t,m,s)$\nobreakdash-nets with $t$ as small as possible ($t$ is called the quality parameter, with lower values corresponding to stronger uniformity of the net/sequence). It has been observed (see e.g.\ \cite[Lemma~1]{NX}) that the existence of a $(t,s)$-sequence implies the existence of $(t,m,s+1)$-nets for all $m > t$. Various constructions have been proposed, the best among them using global function fields. We will use the following upper bound on the lowest possible value of $t$.

\begin{theorem}[Xing and Niederreiter~\cite{XN}] \label{thm:xn}
For every positive integer $s$ there exists a $(t,s)$\nobreakdash-sequence in base $2$ with $t \le 5s - 8\sqrt{\frac{s-1}{3}} - 3$. Moreover, for infinitely many values of $s$, there exists a $(t,s)$\nobreakdash-sequence in base $2$ with $t < 3s$.
\end{theorem}

These upper bounds are not sharp in general. In particular, for low values of $s$, better estimates are known (see~\cite[Table~III]{NX}): e.g., $(t,s)$-sequences in base $2$ with $(t,s) = (0,2),(1,3),(1,4),(2,5),\ldots$ have been constructed (and can be used, as explained below, to get the upper bounds on $h(d)$ in dimensions $d=3,4,5,6$ stated above). However, as $s$ grows, $t$ must grow linearly in $s$. The strongest known lower bound, due to Sch\"urer~\cite{Sc}, is $t > s - (1+o(1))\log_2s$.

Our generalization of Horton's construction to higher dimensions uses $(t,m,s)$-nets and is summarized in the following proposition, proved in the next section.

\begin{proposition} \label{prop:net-to-set-basic}
Let $d \ge 2$ and let $t \le m$ be nonnegative integers so that a $(t,m,d)$-net in base $2$ exists. Then there exists a set $A$ of $2^m$ points in general position in $\mathbb{R}^d$, having no holes of size greater than $2^d(2^{t+d-1}-2^t+1)$.
\end{proposition}

Together with Theorem~\ref{thm:xn}, and the fact that a $(t,s)$-sequence entails $(t,m,s+1)$-nets for all $m > t$, this implies the upper bound on $h(d)$ stated in Theorem~\ref{thm:upper-bound} (with something to spare). The second part of Theorem~\ref{thm:xn} shows that for infinitely many values of $d$, we get an upper bound on $h(d)$ which is exponentially better than stated in Theorem~\ref{thm:upper-bound}. The specific upper bounds on $h(d)$ for low values of $d$ stated above follow by plugging in the parameters of the corresponding known constructions of $(t,s)$\nobreakdash-sequences.

\paragraph{Improvement.} After the original version of the paper was written, we noticed that we may replace
$(t,m,d)$-nets by sets satisfying a weaker condition. For $0 \le \veps < 1$, a non-empty set $X\subseteq [0,1)^d$ is a \emph{$(T,\veps)$-almost net} in base $2$ if
$\abs{X}=2^nT$ for some natural number $n$ and
\[
  (1-\veps)T\leq \abs{X\cap B}\leq (1+\veps)T
\]
for every dyadic box $B$ of volume $2^{-n}$. The following is a generalization of \Cref{prop:net-to-set-basic}.
\begin{proposition} \label{prop:net-to-set}
Let $d \ge 2$ and suppose there exists a $(T,\veps)$-almost net in base $2$ in $[0,1)^d$ of size $2^nT$. Then there exists a set $A$ of $2^nT$ points in general position in $\mathbb{R}^d$, having no holes of size greater than $2^d(2^{d-1}(1+\veps)T-(1-\veps)T+1)$.
\end{proposition}
In \cite{almostnets}, we construct for every natural number $n$ a $(T,1/3)$-almost net in base $2$ in $[0,1)^d$ of size $2^nT$, where $T\leq 900d\log (2d)$. This implies the following improvement:
\begin{theorem}\label{thm:better}
For all sufficiently large $d$ we have $h(d)=O(4^dd\log d)$.
\end{theorem}

\section{Horton-like constructions}
\paragraph{Geometric idea.}
Our construction uses the same basic idea that is used in Horton's construction, and in Valtr's construction.
Namely, if $U\subset \R^d$ is finite, $v\in \R^d$ is arbitrary, and $e\in \R^d$ is a non-zero vector, then from the point of view of $U$, for large values of $t\in \R_+$
the convex hull $\conv (U\cup \{v+te\})$ is almost equal to $\conv(U)+e\R_+$, the Minkowski sum of the set $\conv(U)$ and the ray $e\R_+$. The set $\conv(U)+e\R_+$ has two advantages:
it is independent of $v$ and it is geometrically simpler than $\conv(U\cup \{v+te\})$. We extract the desirable properties into a lemma.

For $U\subset \R^d$, we denote by $\conv U$ its convex hull, by $U^o$ its interior, and by $\convo U$ the interior of its convex hull. For $e \in \R^d$, we write $U + e$ for the translate of the set $U$ by the vector $e$, and $U - e$ is defined similarly. Given a non-zero vector $e\in \R^d$, we denote by $\overline{p}_e$ the projection of the point $p \in \R^d$ on the subspace orthogonal to $e$. We drop the subscript $e$ when it is clear from the context, and use the similar notation $\overline{U}$ for the projection of the set $U$.
\begin{lemma}\label{lem:far_apart}
Suppose $U,V\subset \R^d$ are finite, and $e\in \R^d$ is a non-zero vector. Then there exists a large $t^*=t^*(U,V)$ with the following property.
For all $U'\subseteq U$, $V'\subseteq V$, with $V'\neq \emptyset$, for any point $u \in U$, and every $t\geq t^*$ we have:
\begin{enumerate}[label=(\alph*), ref=(\alph*)]
\item \label{part:closei} if $u\in
(\conv(U')+e\R_+)^o$ then $u \in\convo \bigl(U'\cup (V'+te)\bigr)$, and
\item \label{part:closeiii} if $\overline{u}\in \convo \overline{U'\cup V'}$
then $u\in \bigl(\conv \bigl(U'\cup (V'+te)\bigr)-e\R_+\bigr)^o$.
\end{enumerate}
\end{lemma}
Part \ref{part:closei} of the lemma is illustrated in the figure below.
As the lemma is intuitively plausible, we defer its proof to the end of this section.
\begin{center}
\begin{tikzpicture}
\def\botclip{-0.3}   
\def\topclip{1.8}    
\def\seglen{0.79cm}  
\def\waveleft{1.84}  
\def\waveright{2.7}  

\coordinate (U1) at (0,0);
\coordinate (U2) at (-1,0.6);
\coordinate (U3) at (0.2,1);
\coordinate (V1) at (5,1.5);
\coordinate (V2) at (4.8,0.8);
\begin{scope}
\clip (-2,\botclip) -- (-2,\topclip) -- (\waveleft,\topclip) decorate[decoration={snake,segment length=\seglen}] { -- (\waveleft,\botclip) } -- cycle ;
\coordinate (U1r) at (4,0);
\coordinate (U3r) at (4,1);
\draw[fill=lightgray] (U3) -- (U3r) -- (U1r) -- (U1) -- (U2) -- cycle;
\draw[very thin] (U1) -- (U3);
\fill (U1) circle (2pt);
\fill (U2) circle (2pt);
\fill (U3) circle (2pt);
\draw[fill=white] (-0.6,1.2) circle (2pt);
\draw[fill=white] (-0.2,1.6) circle (2pt);
\draw[fill=white] (0.4,0.5) circle (2pt);
\draw[fill=white] (-0.7,0.12) circle (2pt);
\node at (0.4,0.3) {$\scriptstyle u$};
\end{scope}
\draw[thick] (\waveleft,\topclip) decorate[decoration={snake,segment length=\seglen}] { -- (\waveleft,\botclip) };
\begin{scope}
\clip (\waveright,\topclip) decorate[decoration={snake,segment length=\seglen}] {-- (\waveright,\botclip)} -- (6,\botclip) -- (6,\topclip) -- cycle;
\coordinate (V1l) at (1,1.5);
\coordinate (V2l) at (1,0.8);
\draw[fill=lightgray] (V1) -- (V2) -- (V2l) -- (V1l) -- cycle ;
\fill (V1) circle (2pt);
\fill (V2) circle (2pt);
\draw[fill=white] (5.2,0.5) circle (2pt);
\draw[fill=white] (4.7,1.7) circle (2pt);
\draw[fill=white] (4.1,0.3) circle (2pt);
\draw[fill=white] (4.6,1) circle (2pt);
\draw[fill=white] (4.3,1.3) circle (2pt);
\end{scope}
\draw[thick] (\waveright,\topclip) decorate[decoration={snake,segment length=\seglen}] { -- (\waveright,\botclip) };
\node at (6.3,1) {$e$};
\draw[->] (6.1,0.8) -- (6.5,0.8);
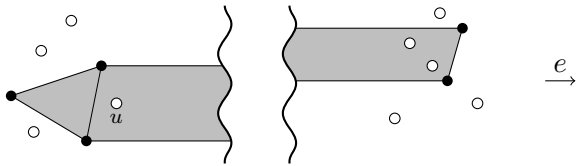
\node at (2.3,-0.8) {\textbf{Figure 1:} The set $U$ is on the left, the set $V+te$ is on the right.};
\node at (2.2,-1.35) {The black points are the elements of $U'$ and $V'+te$ respectively.};
\node at (0.4,-1.9) {The convex hull of $U'\cup (V'+te)$ is in gray.};
\end{tikzpicture}
\end{center}
We will use the following consequence of \Cref{lem:far_apart}.
\begin{lemma}\label{lem:geoinduct}
Suppose $U,V,W\subset \R^d$ are finite, and $e\in \R^d$ is a non-zero vector. Let $t \ge t^*(U,V)$, and $t' \ge t^*\bigl(U\cup(V+te),W\bigr)$
with $t^*$ as in \Cref{lem:far_apart}. Assume that $S\subseteq U\cup (V+te)$
and $u\in U$ satisfy
\begin{itemize}
\item the intersection $S\cap (V+te)$ is non-empty, and
\item $\overline{u}\in \convo \overline{S}$.
\end{itemize}
Then $u\in \convo (S\cup \{w\})$ for every $w\in W-t'e$.
\end{lemma}
Like the proof of \Cref{lem:far_apart}, we defer the proof of the preceding lemma to the end of the section.

We apply the construction in \Cref{lem:far_apart} repeatedly. We start with the one-element set containing the origin. At each step, we choose a direction $e$ and replace the previously
constructed set $U$ by $U\cup (U+te)$ for suitably large $t$. The directions are chosen among
the standard basis vectors as follows: for the first $m$ steps we choose $e_1$ and apply the lemma relative to $\R^1$, for the next
$m$ steps we choose $e_2$ and apply the lemma relative to $\R^2$, and so forth, ending with $m$ steps when we choose $e_d$ and work in $\R^d$. Each point of the resulting
set is of the form
\[
  P(\ba)\eqdef \sum_{\substack{i\in[d]\\j\in[m]}} a^i_j t_{i,j}e_i,
\]
where $\ba=(a^1,a^2,\dotsc,a^d)\in (\{0,1\}^m)^d$, and \[0 \ll t_{1,m}\ll t_{1,m-1}\ll \dotsb \ll t_{1,1}\ll t_{2,m}\ll t_{2,m-1}\ll \dotsb \ll t_{2,1}\ll  \dotsb \dotsb \ll t_{d,m}\ll t_{d,m-1}\ll \dotsb \ll t_{d,1}\] with the meaning
of $\ll$ being supplied iteratively by \Cref{lem:far_apart}. Note that we chose to parameterize the points so that the last entry of $a^i$ corresponds to the first step of the construction in direction $e_i$, etc. We may also assume that each next $t_{i,j}$ is at least double the preceding one. This way the order between the $i$'th coordinate values of two points $P(\ba)$ and $P(\bb)$ is determined by the lexicographic order between $a^i$ and $b^i$. Our Horton-like construction will consist
of appropriately chosen points of the form $P(\cdot)$.

\paragraph{Good sets.}
We next describe a sufficient condition on a set $Y\subseteq (\{0,1\}^m)^d$ that ensures the absence of large holes in~$P(Y)$.

We call $a\in \{0,1\}^k$ \emph{a binary sequence of length $k$} and write $k = \len a$. We denote the concatenation of sequences $a$ and $b$ by $ab$.
We write $a\prf b$ if $a$ is a prefix of $b$.
For $a\in \{0,1\}^{k}$, we denote by $\da$ the sequence of length $k-1$ obtained
from $a$ by removing the last element.

\begin{definition}\label{def:good}
We say that a set $Y \subseteq (\{0,1\}^m)^d$ is \emph{$q$-good} if every pair of distinct points $\bx, \by \in Y$ satisfies $x^i \ne y^i$ for all $i \in [d]$, and the following holds true. For every $d-1$ binary sequences $a^2,\dotsc,a^d$ (possibly of different lengths) and every $(q+1)$-element set $Z\subseteq Y$ obeying the condition
\begin{enumerate}[label=(C), ref=(C),labelindent=!,leftmargin=*,wide]
\item \label{cond:almostprefix} for each $i\in \{2,3,\dotsc,d\}$, all $\bz\in Z$ satisfy $\da^i\prf z^i$,
\end{enumerate}
there is $\by\in Y$ such that $a^i\prf y^i$ for all $i\in \{2,3,\dotsc,d\}$ and $\min \{z^1 : \bz\in Z\} < y^1 <\max\{z^1 :\bz\in Z\}$ in the lexicographic order.
\end{definition}
We shall see below that any $(T,\veps)$-almost net in base~$2$ can be turned into a \linebreak$(2^d(1+\veps)T-2(1-\veps)T+2)$-good set. In particular, since a $(t,m,d)$-net in base~$2$ is also a $(2^t,0)$-almost net in base~$2$, any $(t,m,d)$-net in base~$2$ can be turned into a $(2^{t+d} - 2^{t+1} + 2)$-good set.

\begin{definition}\label{def:holefree}
Given a finite set of points $V\subseteq \mathbb{R}^d$, we say that $V$ is \emph{$\ell$-hole-free} if for any $\ell$ points $v_1,v_2,...,v_{\ell}\in V$, there is a point $v\in V$ in the interior of $\conv\{v_1,v_2,...,v_{\ell}\}$.
\end{definition}
If the set $V$ is in general position, the definition agrees with the usual definition of a set without $\ell$-holes. The advantage
of this definition is its robustness: every sufficiently small perturbation of an $\ell$-hole-free set, which is not
necessarily in general position, is again $\ell$-hole-free.

\begin{theorem}\label{fcntogeo}
Let $d \ge 2$, $m$ and $q$ be positive integers, and suppose that $Y\subseteq (\{0,1\}^m)^d$ is $q$-good. Then the set $P(Y)$ is $(2^{d-1}q+1)$\nobreakdash-hole-free.
\end{theorem}

By the remark following \Cref{def:holefree}, we do not need
to worry about general position. So, \Cref{fcntogeo} gives us a purely combinatorial way to construct $\ell$\nobreakdash-hole\nobreakdash-free sets.

\begin{proof}[Proof of \Cref{fcntogeo}]
Let $U\subseteq Y$ be an arbitrary set of size $|U|> 2^{d-1}q$. We must show that there is a $\by\in Y$ such that $P(\by)\in\convo P(U)$.

We shall define sets $U_d\supseteq U_{d-1}\supseteq U_{d-2}\supseteq \dotsb \supseteq U_1$ and binary sequences $a^{d},a^{d-1},\ldots,a^2$ inductively. We begin by setting $U_d\eqdef U$.
Suppose $i > 1$ and $U_i$ has been defined. Denote by $U_i^i$ the set $\{x^i : \bx\in U_i\}$. Let $b^i$ be the longest
binary sequence that is a prefix of all elements of $U_i^i$, and
let $\alpha_i$ be an element of $\{0,1\}$ which maximizes the size of
\[
  U_{i-1}\eqdef \{\bx\in U_i : b^i\,\alpha_i\prf x^i \};
\]
in case of a tie, we pick $\alpha_i$ arbitrarily. Note that $\abs{U_{i-1}}\geq \abs{U_i}/2$.
Let $\beta_i\eqdef 1-\alpha_i$. We then set $c^i$ to be the longest sequence such that $b^i\,\alpha_i\, c^i$ is a prefix of all elements of $U_{i-1}^i\eqdef\{x^i : \bx\in U_{i-1}\}$ and define
\[
  a^i\eqdef b^i\,\alpha_i\,c^i\,\beta_i.
\]
It is clear that $a^i$ satisfies \ref{cond:almostprefix} for $Z = U_{i-1}$.

This way we obtain a nested sequence $U_1\subseteq U_2\subseteq \dotsb \subseteq U_d$ with $\abs{U_1}>q$. Since $Y$ is $q$-good,
and $a^2,\dotsc,a^d$ and $U_1$ satisfy condition~\ref{cond:almostprefix} in \Cref{def:good}, there exist $\by\in Y$, $\bx_{\text{small}},\bx_{\text{big}}\in U_1$ satisfying $a^i\prf y^i$ for all $i\in \{2,3,\dotsc,d\}$
as well as $x_{\text{small}}^1<y^1<x_{\text{big}}^1$ (in the lexicographic ordering). We claim that $P(\by)\in \convo P(U)$.

To prove this claim, we will show by induction on $i=1,2,\dotsc,d$ that
\[  \pi_i(P(\by))\in \convo \pi_i(P(U_i)),\]
where $\pi_i\colon \R^d\to \R^i$ is the projection map onto the first $i$ coordinates. The base case $i=1$ holds because of $x_{\text{small}}^1<y^1<x_{\text{big}}^1$. Suppose that $i>1$.
There are two (similar) cases depending on the value of $\alpha_i$. Suppose first that $\alpha_i=1$.
We apply \Cref{lem:geoinduct} in $\R^i$ using the vector $e_i$, with
\begin{align*}
\{\pi_i(P(\bx)) : b^i\,1\,c^i\, 1 \prf x^i,\ \bx\in (\{0,1\}^m)^d\}&\text{ in place of }V+te,\\
\{\pi_i(P(\bx)) : b^i\,1\,c^i\, 0 \prf x^i,\ \bx\in (\{0,1\}^m)^d\}&\text{ in place of }U,\\
\{\pi_i(P(\bx)) : b^i\, 0\,\phantom{c^i\, 1} \prf x^i,\ \bx\in (\{0,1\}^m)^d\}&\text{ in place of }W-t'e,
\end{align*}
and with $S=\pi_i(P(U_{i-1}))$, $u=\pi_i(P(\by))$, and $w=\pi_i(P(\bx))$ for some $\bx \in U_i$ such that $b^i\,0\prf x^i$ (such $\bx$ exists by the maximality of~$b^i$).
Note that $S\cap (V+te)$ is non-empty by the maximality of~$c^i$, and $\overline{u} \in \convo \overline{S}$ holds by the induction hypothesis.
Therefore we deduce from \Cref{lem:geoinduct} that $u \in \convo(S\cup \{w\}) \subseteq \convo \pi_i(P(U_i))$, as required. The case when $\alpha_i = 0$ is treated similarly
by exchanging the roles of $0$'s and $1$'s, and replacing the vector $e_i$ by $-e_i$.
\end{proof}

\paragraph{Good sets from $(T,\veps)$-almost nets.} Here we show how to transform a $(T,\veps)$-almost net $X\subseteq [0,1)^d$ of size $2^nT$ into a good set $Y \subseteq (\{0,1\}^m)^d$ for $m=n+\lceil\log_2 T\rceil+1$. Fix $i \in [d]$. For $x = (x_1,\dotsc,x_d) \in X$, let $y_i$ be the unique nonnegative integer such that $y_i \le x_i2^m < y_i + 1$. Let $\widetilde{y}_i \in \{0,1\}^m$ be the $m$-digit binary representation of $y_i$. Applying the definition of a $(T,\veps)$-almost net to the sub-boxes of the form $B=[0,1)^{i-1}\times  [\frac{b}{2^{n}}, \frac{b + 1}{2^{n}}) \times [0,1)^{d-i}$, we know that there are between $(1-\veps)T$ and $(1+\veps)T$ points $x$ in $X$ for which the corresponding $\widetilde{y}_i$ has any given prefix of length $n$.
By suitably changing, if necessary, the last $\lceil\log_2 T\rceil+1$ entries of $\widetilde{y}_i$ we obtain $y^i \in \{0,1\}^m$ so that the mapping $x \mapsto y^i$ is injective. Doing this for each $i \in [d]$, we transform every $x$ in $X$ into a $\by = (y^1,y^2,\dotsc,y^d)$ in $(\{0,1\}^m)^d$, so that the resulting set $Y \subseteq (\{0,1\}^m)^d$ satisfies the requirement in \Cref{def:good} that its elements should differ for all $i \in [d]$. Moreover, the definition of a $(T,\veps)$-almost net implies that for any $d$ binary sequences $a^1,a^2,\dotsc,a^d$ with $\sum_{i=1}^d \len a^i = k \le n$, the set \[I(a^1,a^2,\dotsc,a^d)\eqdef \{ \by\in Y : a^i\prf y^i\text{ for all }i\in [d]\}\] has size between $2^{n-k}(1-\veps)T$ and $2^{n-k}(1+\veps)T$. We call such a set $Y$ a \emph{binary $(T,\veps)$-almost net} of size $2^nT$.

The next result, together with \Cref{fcntogeo},
implies \Cref{prop:net-to-set}, which was announced in the introduction, and hence \Cref{thm:better}. The case $\veps = 0$ yields \Cref{prop:net-to-set-basic}, and hence \Cref{thm:upper-bound}.
\begin{proposition}\label{tmstofcn}
If $Y\subseteq (\{0,1\}^m)^d$ is a binary $(T,\veps)$-almost net then $Y$ is $\lfloor 2^d(1+\veps)T-2(1-\veps)T+2 \rfloor$-good.
\end{proposition}
\begin{proof}
Suppose that the binary sequences $a^2,\dotsc,a^d$ and the set $Z\subseteq Y$ with size \linebreak$|Z| > 2^d(1+\veps)T - 2(1-\veps)T + 2$ satisfy condition~\ref{cond:almostprefix} in \Cref{def:good}.
By condition~\ref{cond:almostprefix} we have, using the notation introduced above, $Z \subseteq I(\emptyset,\da^2,\dotsc,\da^d)$. As $|Z| > 2^{d-1}(1+\veps)T$ and
$|I(\emptyset,\da^2,\dotsc,\da^d)| \le (1+\veps)T\max(1,2^{n - \sum_{i=2}^d \len \da^i})$, we conclude that $d-1 < n - \sum_{i=2}^d \len \da^i$ and hence $\sum_{i=2}^d \len a^i=\sum_{i=2}^d \len \da^i +d-1< n$. Thus the quantity $r\eqdef\nobreak n-\sum_{i=2}^d \len a^i$ is positive. Given a sequence $a\in \{0,1\}^r$, consider the sets $B(a)\eqdef I(a,a^2,\dotsc,a^d)$ and
$\hat{B}(a)\eqdef I(a,\da^2,\dotsc,\da^d)$. From the discussion above we know that
$\abs{B(a)}\geq (1-\veps)T$ and $\abs{\hat{B}(a)}\leq 2^{d-1}(1+\veps)T$ for every $a \in \{0,1\}^r$. From condition~\ref{cond:almostprefix} we know also that $Z\subseteq \bigcup_{a\in\{0,1\}^r} \hat{B}(a)$.

Our aim is to find $\by \in Y$ that is contained in some $B(a)$ and whose first coordinate is sandwiched
between the first coordinates of two elements in $Z$.

Suppose first that $Z\cap \hat{B}(a)$ is non-empty for three (or more) distinct sequences $a\in \{0,1\}^r$, say for $a^{(1)},a^{(2)},a^{(3)}$.
We may assume that, of the three, $a^{(1)}$ is the lexicographically smallest and $a^{(3)}$ is the lexicographically
largest. Then we may pick $\by$ to be any element of $B(a^{(2)})$, for its first coordinate is
between those of elements in $Z\cap \hat{B}(a^{(1)})$ and in $Z\cap \hat{B}(a^{(3)})$.

So, we may assume that $Z$ is entirely contained in $\hat{B}(a)\cup \hat{B}(a')$ for some pair $a,a'\in \{0,1\}^r$. Then either
$\hat{B}(a)$ or $\hat{B}(a')$ contains more than $2^{d-1}(1+\veps)T-(1-\veps)T+1$ elements of $Z$. By size considerations, at least $2$ of them must be in the respective $B(\cdot)$-set,
and at most one of the $2$ is extremal in $Z$, so choosing the other one as our $\by$ works.
\end{proof}

\paragraph{Proofs of the geometrical lemmas.} It remains to prove \Cref{lem:far_apart,lem:geoinduct}.
\begin{proof}[Proof of \Cref{lem:far_apart}]
Because there are only finitely many subset pairs $(U',V')$ and points $u \in U$, it suffices to prove the assertion for any one such choice.
We may then pick the largest $t^*$ over all choices of $(U',V')$ and $u$.\smallskip

\textit{Proof of part \ref{part:closei}.}
Pick $v\in V'$ arbitrarily. Let $u\in (\conv (U')+e\R_+)^o$ be arbitrary.
Let $B(u,\veps)$, with $\veps>0$, be a closed ball around $u$ that is contained in
$\conv (U')+e\R_+$.

Assume, for contradiction's sake, that $u\notin \convo(U'\cup\{v+te\})$.
Then there is a hyperplane through $u$ such that the convex set
$\conv(U'\cup\{v+te\})$ lies entirely on one of its sides.
Pick a unit normal vector $w$ to this hyperplane, such that the halfspace $H\eqdef \{ x : \langle w,x-u\rangle > 0\}$
does not meet $\conv(U'\cup\{v+te\})$. Consider the point $\widetilde{u}\eqdef u+\veps w$, and note that $\widetilde{u}\in H$.

Since $\dist(u,\widetilde{u})= \veps$, it follows that $\widetilde{u}\in \conv(U')+e\R_+$,
and so we may write $\widetilde{u}=u_0+et_0$ with $u_0\in \conv U'$ and $t_0\in \R_+$.
Define points $p\eqdef \frac{t_0}{t_0-t}v+\frac{t}{t-t_0}u_0$
and $u'\eqdef (t_0/t)(v+te)+(1-t_0/t)u_0$.
We may pick $t^*$ large enough so that
$\dist(p,u_0)<\veps$ for $t\geq t^*$. Since $\widetilde{u}=(t_0/t)(v+te)+(1-t_0/t)p$, it then
follows that $\dist(\widetilde{u},u')<\veps$, and hence $u'\in H$. Since $u'\in \conv(U'\cup\{v+te\})$,
this contradicts the definition of~$H$.\smallskip


\textit{Proof of part \ref{part:closeiii}.} We first note that it suffices to show that
$u\in \conv \bigl(U'\cup (V'+te)\bigr)-e\R_+$, for we may then apply this to all points in
a sufficiently small neighborhood of $u$ to conclude that in fact $u\in \bigl(\conv \bigl(U'\cup (V'+te)\bigr)-e\R_+\bigr)^o$.

As $\overline{u}$ is in the interior of $\conv \overline{U'\cup V'}$,
we may write it as a convex combination, in which the coefficients of every point in $\overline{U'}$ and of every
point in $\overline{V'}$ are non-zero. Indeed, for sufficiently small $\veps>0$, the point
$\overline{u}_{\veps}\eqdef (1+\veps)\overline{u}-\frac{\veps}{\abs{U'}+\abs{V'}}\sum_{u'\in U'}\overline{u\raisebox{-2pt}{$'$}}-
\frac{\veps}{\abs{U'}+\abs{V'}}\sum_{v'\in U'}\overline{v\raisebox{-2pt}{$'$}}$
is in $\conv \overline{U'\cup V'}$. Writing $\overline{u}_{\veps}$ as a convex combination
of the points in $\overline{U'\cup V'}$, and rearranging, we obtain
an expression for $\overline{u}$ as a convex combination with strictly positive coefficients.

Fix such a convex combination, say $\overline{u}=\sum_{u'\in U'} \alpha_{u'} \overline{u\raisebox{-2pt}{$'$}}+\sum_{v'\in V'} \beta_{v'} \overline{v\raisebox{-2pt}{$'$}}$.
Since the $\beta$'s are positive, we may choose $t^*$ large enough so that for $t \ge t^*$, the convex combination
$\sum_{u'\in U'} \alpha_{u'} u'+\sum_{v'\in V'} \beta_{v'} (v'+te)$ is above $u$ in the direction~$e$, and so
$u\in \conv \bigl(U'\cup (V'+te)\bigr)-e\R_+$.
\end{proof}
\begin{proof}[Proof of \Cref{lem:geoinduct}]
From \Cref{lem:far_apart}\ref{part:closeiii} applied to the sets $U'=S\cap U$ and $V'=(S-te)\cap V$,
using $\overline{u}\in \convo \overline{S}$ we deduce that $u\in(\conv S-e\R_+)^o$ (note that $V'\neq \emptyset$ because $S\cap (V+te)$ is non-empty).
Then, from \Cref{lem:far_apart}\ref{part:closei} applied to the
sets $U\cup (V+te)$ and $W$ in place of $U$ and $V$, the direction $-e$ in place of $e$, with $U'=S$ and $V'=\{w+t'e\}$,
we obtain the desired conclusion.
\end{proof}

\section{Problems and remarks}
\begin{itemize}
\item We suspect that every large enough set in general position in $\R^d$ contains an exponentially large hole. However,
we were unable to improve Valtr's bound $h(d)\geq 2d+1$. We do have an argument (details omitted) showing that constructions along the
  lines of Horton's, Valtr's and ours cannot avoid exponentially large holes:
  if a set in $\R^d$ consists of points of the form $P(\ba)$ and is large enough,
  then it contains a hole of size $2^d$.

\item Let $f_{d,\ell}(n)$ be the least number of $\ell$-holes in an $n$-point set in general position in $\R^d$.
It is possible to give lower bounds on $f_{d,\ell}(n)$. First, for $\ell\leq h(d)$, we may cut the $n$-point set
into linearly-many equally large pieces by parallel hyperplanes. If each piece is large enough, then it contains
an $\ell$-hole, and so $f_{d,\ell}(n)=\Omega(n)$ in this case.

Second, $f_{d+1,\ell+1}(n+1)\geq \frac{n+1}{\ell+1}\cdot f_{d,\ell}(\lceil n/2 \rceil)$ holds. Indeed, suppose $P\subset \R^{d+1}$ is in general position
and $p\in P$ is arbitrary. Pick any hyperplane that passes only through $p$, and push it slightly towards
the side containing more points of~$P$. Consider the central projection towards $p$ to the hyperplane of points
on this larger side; we may think of it as a set in $\R^d$. Every $\ell$-hole in this set entails an $(\ell+1)$-hole in~$P$.
As an $(\ell+1)$-hole arises in this manner at most $\ell+1$ times, the bound follows.

Taken together with the known lower bounds on $f_{2,\ell}(n)$ and with the aforementioned bound of Valtr, these two observations yield $f_{d,\ell}(n)=\Omega(n^d)$ for $\ell=d+1,d+2$,
$f_{d,d+3}(n)=\Omega(n^{d-1}\log^{4/5}n)$, $f_{d,d+4}(n)=\Omega(n^{d-1})$, and $f_{d,d+k}(n)=\Omega(n^{d-k+2})$ for $k=5,\dotsc,d+1$.

\item It would be interesting to characterize large sets that contain no holes of some fixed size. In this connection we conjecture that,
for each $n,\ell\in \N$, every sufficiently large $\ell$-hole-free set in general position in $\R^2$ contains an $n$-point subset
whose order type is the same as that of an $n$-point Horton set.
\end{itemize}

\bibliographystyle{alpha}

\bibliography{holes}
\end{document}